\documentclass[11pt,a4paper,twoside]{amsart}

\usepackage[english]{babel}

\usepackage[pagewise, mathlines, displaymath]{lineno}

\usepackage{amssymb}
\usepackage{amsmath}
\usepackage{amsthm}
\usepackage[dvips]{graphicx}
\usepackage[textsize=tiny]{todonotes} 
\usepackage{enumerate}
\usepackage{mathabx}
\usepackage{bbm}
\usepackage[noadjust]{cite}
\usetikzlibrary{patterns}

\theoremstyle{plain}
\newtheorem{thm}{Theorem}[section]
\newtheorem{prop}[thm]{Proposition}

\newtheorem{lemma}[thm]{Lemma}

\theoremstyle{definition}
\newtheorem{defin}[thm]{Definition}

\theoremstyle{remark}
\newtheorem{rmk}[thm]{Remark}

\numberwithin{equation}{section}
\numberwithin{figure}{section}

\newcommand{\abs}[2][{}]{\lvert{#2}\rvert_{#1}}    

\newcommand{\normsymb}{\|}

\newcommand{\norm}[2]{\normsymb{#1}\normsymb_{#2}}  

\newcommand{\RR}{\mathbb{R}} 
\newcommand{\CC}{\mathbb{C}} 
\newcommand{\NN}{\mathbb{N}} 
\newcommand{\ZZ}{\mathbb{Z}} 

\newcommand{\cH}{{\mathcal H}}

\newcommand{\cU}{{\mathcal U}}
\newcommand{\cT}{{\mathcal T}}
\newcommand{\cP}{{\mathcal P}}

\newcommand{\fa}{{\mathfrak a}}

\newcommand{\fh}{{\mathfrak h}}

\DeclareMathOperator{\dd}{d\!}  
\newcommand{\ee}{\mathrm e}     
\newcommand{\EE}{\mathsf{E}}    
\newcommand{\ii}{\mathrm i}     
\DeclareMathOperator{\Dom}{{\mathcal D}} 
\renewcommand{\epsilon}{\varepsilon}
\DeclareMathOperator{\supp}{supp}
\DeclareMathOperator{\dist}{dist}
\DeclareMathOperator{\diag}{diag}

\title[The reflection principle in the control problem]{The reflection principle in the control problem of the heat equation}


\author[M.~Egidi]{Michela Egidi}
\address[M.~Egidi]{Ruhr Universit\"at Bochum, Fakult\"at f\"ur Mathematik, D-44780 Bochum, Germany}
\email{michela.egidi@ruhr-uni-bochum.de}

\author[A.~Seelmann]{Albrecht Seelmann}
\address[A.~Seelmann]{Technische Universit\"at Dortmund, Fakult\"at f\"ur Mathematik, D-44221 Dortmund, Germany}
\email{albrecht.seelmann@math.tu-dortmund.de}

\date{\today}

\begin{document}

\begin{abstract}
 We consider the control problem for the generalized heat equation for a Schr\"odinger operator on a domain with a reflection
 symmetry with respect to a hyperplane. We show that if this system is null-controllable, then so is the system on its respective
 parts and the corresponding control cost does not exceed the one on the whole domain. 
 As an application, we obtain null-controllability results for the heat equation on half-spaces, orthants, and sectors of angle
 $\pi/2^n$. As a byproduct, we also obtain explicit control cost bounds for the heat equation on certain triangles and corresponding
 prisms in terms of geometric parameters of the control set.
\end{abstract}

\maketitle


\section{Introduction and main result}\label{sec:intro} 

Let $\Omega \subset \RR^d$ be open, $d \in \NN$, and let $A\colon\Omega\rightarrow \RR^{d\times d}$ be measurable with $A(x)$ a
symmetric matrix for almost every $x\in\Omega$. Suppose, in addition, that there exist $\theta_1,\theta_2>0$ such that 
\begin{equation}\label{eq:ellipticity}
 \theta_1 \norm{\xi}{\CC^d}^2
 \leq
 \langle A(x)\xi, \xi \rangle_{\CC^d}
 \leq
 \theta_2 \norm{\xi}{\CC^d}^2
 \qquad \forall \, \xi\in\CC^d,\text{ a.e.~}x\in\Omega,
\end{equation}
and let $V\in L^\infty(\Omega)$ be real-valued. We denote by $H_\Omega^D=H_\Omega^D(A,V)$ and $H_\Omega^N=H_\Omega^N(A,V)$ the
Dirichlet and Neumann realizations of the differential expression
\[
-\nabla\cdot (A\nabla) + V
\]
as self-adjoint lower semibounded operators on $L^2(\Omega)$ defined via their quadratic forms with form domain $H_0^1(\Omega)$ and
$H^1(\Omega)$, respectively. For details of this construction we refer the reader to the discussion in
Section~\ref{sec:proof-main-thm} below.

Let $T>0$, $\omega\subset\Omega$ be a measurable subset, and $\bullet\in \{D, N\}$. We consider the heat-like system
\begin{equation}\label{eq:heat-like-eq} 
 \partial_t u(t) + H^\bullet_\Omega u(t) = \chi_{\omega}v(t) \quad\text{ for }\quad 0<t<T,\quad u(0)=u_0,
\end{equation}
with $u_0\in L^2(\Omega)$ and $v\in L^2((0,T),L^2(\Omega))$. Here, $\chi_\omega$ denotes the characteristic function of $\omega$,
and we call $\omega$ a~\emph{control set} for the system~\eqref{eq:heat-like-eq}.

System~\eqref{eq:heat-like-eq} is said to be~\emph{null-controllable in time $T>0$} if for every initial data $u_0\in L^2(\Omega)$  
there exists a control function $v\in L^2((0,T),L^2(\Omega))$ such that the mild solution of~\eqref{eq:heat-like-eq} satisfies
$u(T)=0$, that is,
\begin{equation}\label{eq:nullcontrol}
 \ee^{-TH_\Omega^\bullet}u_0 + \int_0^T \ee^{-(T-s)H_\Omega^\bullet}\chi_\omega v(s) \dd s = 0
 \quad\text{ in }\quad L^2(\Omega).
\end{equation}
In this case, the quantity
\begin{equation}\label{eq:control_cost}
 C_T:=\sup_{\norm{u_0}{L^2(\Omega)}=1}\inf\{\norm{v}{L^2((0,T),L^2(\omega))} \colon v\text{ satisfies }\eqref{eq:nullcontrol} \}
\end{equation}
is called~\emph{control cost}.

For~\emph{bounded} domains $\Omega$, there is already a rich literature on null-control\-lability results for
system~\eqref{eq:heat-like-eq} with $H_\Omega^\bullet$ being the Laplace or a Schr\"odinger operator with an open or measurable
control set $\omega$, see, e.g.,~\cite{lebeau-robbiano-95,TenenbaumT-07,ApraizEWZ-14,NTTV18-preprint,EV18}. On the other hand,
null-controllability of these systems on~\emph{unbounded} domains has been an issue of growing interest only recently, see for
example~\cite{ENSTTV20,EV18,NTTV18-preprint} and the references therein.
As far as bounds on the control cost are concerned, one classically asks for the type of time dependency, which has been largely
explored in~\cite{FI96,Phung04,TenenbaumT-07,DZZ08,miller:10,EZ11,LL12,NTTV18-preprint}, see also the references therein. However,
the dependency of the control cost on geometric parameters of $\Omega$ and $\omega$ has only recently been addressed
in~\cite{ENSTTV20,Egi18,EV18,NTTV18-preprint}, see also~\cite{miller:04} for previous results. Especially the
work~\cite{NTTV18-preprint} focused on these dependencies, which have been exploited in several asymptotic regimes there in order
to discuss homogenization.

The aim of the present paper is to prove null-controllability results for the above situation on some new unbounded domains, such as
sectors of certain angles, half-spaces, and orthants, with bounds on the associated control cost that are fully or partially
consistent with the known ones on cubes and the full space. As a byproduct, we also obtain explicit control cost bounds in terms of
only geometric parameters of the control set for the corresponding system on some bounded domains with flat boundary that were not
accessible before. We believe that the lack of the control cost's dependency on the domain may be an effect of the flatness of its
boundary. However, it is not our current scope to explore this matter in depth, but merely to add examples to support this
conjecture. All these results are discussed in detail in Section~\ref{sec:application} below as applications of a reflection
principle which we now describe.

The key idea in our considerations is to relate system~\eqref{eq:heat-like-eq} for certain domains $\Omega$ to a corresponding
symmetrized system on a larger domain, where null-controllability results are available. In order to make this precise, let
$M\colon\RR^d\rightarrow\RR^d$ be the reflection with respect to the first coordinate, that is,
$M(x_1,\ldots,x_d)=(-x_1,x_2,\ldots, x_d)$. We suppose that $\Omega$ is contained in the half-space $(0,+\infty)\times\RR^{d-1}$
such that $\Omega = \tilde{\Omega}\cap((0,+\infty)\times\RR^{d-1})$ for some open set $\tilde\Omega\subset\RR^d$ with
$\Gamma := \tilde{\Omega} \cap (\{0\} \times \RR^{d-1}) \neq \emptyset$ and which is symmetric with respect to the reflection $M$,
that is, $M(\tilde\Omega)=\tilde{\Omega}$. In particular, we have $\tilde{\Omega}=\Omega\cup\Gamma\cup M(\Omega)$,
cf.~Figure~\ref{fig:reflection}.

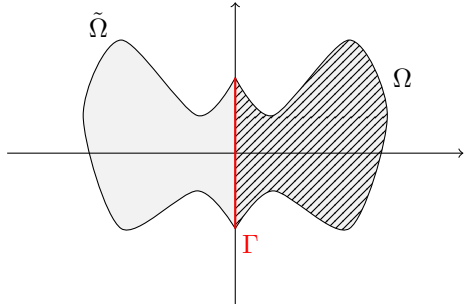
\begin{figure}[htb]
\begin{tikzpicture}[scale=1]
\filldraw[color=black, fill=black!5] plot [smooth] coordinates {(0,-1) (-0.5, -0.5) (-1.5, -1) (-2, 0.5) (-1.5, 1.5) (-0.5, 0.5) (0,1)};
\filldraw[color=white, fill=black!5] plot [smooth] coordinates {(0,-1) (0.5, -0.5) (1.5, -1) (2, 0.5) (1.5, 1.5) (0.5, 0.5) (0,1)};
\filldraw[ pattern=north east lines] plot [smooth] coordinates {(0,1) (0.5,0.5) (1.5,1.5) (2,0.5) (1.5,-1) (0.5,-0.5) (0,-1)};
 \node at (-1.8,1.7) {\small$\tilde\Omega$};
 \node at (2.2, 1) {\small$\Omega$};
 \draw[->] (-3, 0) -- (3,0);
 \draw[->] (0, -2) -- (0,2);
 \draw[red,thick] (0,1) -- (0,-1); \node[red] at (0.2,-1.2) {\small$\Gamma$};
\end{tikzpicture}
\caption{$2$-dimensional example of $\tilde\Omega$, $\Omega$, and $\Gamma$.}\label{fig:reflection}
\end{figure}

Let $\tilde{A}\colon\tilde{\Omega}\rightarrow\RR^{d\times d}$ and $\tilde{V}\colon\tilde{\Omega}\rightarrow\RR$ be measurable with
\[
\tilde{A}(x) = \left\{
\begin{array}{ll}
A(x), & x\in \Omega\\
U (A\circ M)(x) U, & x\in M(\Omega)
\end{array}\right.
,\qquad U:=\diag(-1,1,\ldots,1),
\]
and
\[
\tilde{V}(x)=\left\{
\begin{array}{ll}
V(x), & x\in\Omega \\
(V\circ M)(x), & x\in M(\Omega) 
\end{array}\right..
\]
Observe that by construction $\tilde{A}(x)$ is for almost every $x\in\tilde{\Omega}$ a symmetric matrix
satisfying~\eqref{eq:ellipticity} with the same constants $\theta_1,\theta_2$.

Let the self-adjoint operators $H_{\tilde\Omega}^\bullet=H_{\tilde\Omega}^\bullet(\tilde{A},\tilde{V})$ on $L^2(\tilde{\Omega})$
associated with the differential expression
\[
-\nabla\cdot (\tilde{A}\nabla) + \tilde{V}
\]
be defined analogously to $H_\Omega^\bullet(A,V)$ above. Set $\tilde{\omega}:= \omega\cup M(\omega)\subset \tilde{\Omega}$, and
consider the corresponding symmetrized system 
\begin{equation}\label{eq:system-big-domain} 
\partial_t \tilde{u}(t) + H_{\tilde\Omega}^\bullet \tilde{u}(t) = \chi_{\tilde{\omega}}\tilde{v}(t) \qquad\text{ for }\quad 0<t<T,\quad
\tilde{u}(0)=\tilde{u}_0,
\end{equation}
with $\tilde{u}_0\in L^2(\tilde\Omega)$, $\tilde{v}\in L^2((0,T),L^2(\tilde\Omega))$, and control set $\tilde{\omega}$. The notions
of null-controllability and control cost for this system carry over verbatim.

The main result of the present paper now reads as the following theorem, which can actually be formulated also for much more general
situations than the ones we discuss here, cf.~Section~\ref{sec:abstract-result} below.

\begin{thm}\label{thm:null-controllability}
 If the system~\eqref{eq:system-big-domain} is null-controllable in time $T>0$ with control cost $\tilde{C}_T$, then also
 system~\eqref{eq:heat-like-eq} is null-controllable in time $T>0$ with control cost $C_T\leq \tilde{C}_T$.
\end{thm}

The restriction to the reflection symmetry with respect to the hyperplane $\{0\}\times\RR^{d-1}$ in
Theorem~\ref{thm:null-controllability} is not essential. Indeed, by rotating the whole system, we can deal with reflection
symmetries with respect to any hyperplane in $\RR^d$. This way, the above theorem allows us to infer null-controllability on the
respective parts of a domain with a reflection symmetry if the system on the whole domain is null-controllable.

The rest of the paper is organized as follows. In Section~\ref{sec:application} we discuss the null-controllability results and
control cost bounds announced above for the case where $A(x)$ is the identity matrix as consequences of
Theorem~\ref{thm:null-controllability}. Here we will make explicit use of recent results
from~\cite{EV18,NTTV18-preprint} for cubes and the whole space. The assumption on the essential boundedness
of the potential $V$ is tailored towards these applications but it is not essential for the general argument.
In Section~\ref{sec:abstract-result} we prove an abstract result which is not only the core of the proof of
Theorem~\ref{thm:null-controllability} but also promises to have a much broader range of application and is therefore of its own
interest, see, e.g., Remarks~\ref{rmk:fractionalLaplacian-BanachSpaces},~\ref{rmk:modification}, and~\ref{rmk:triangle} below.
Section~\ref{sec:proof-main-thm} deals with the proof of Theorem~\ref{thm:null-controllability} based on the abstract result of
the preceding section. Finally, Appendix~\ref{sec:intParts} provides an integration by parts formula used in
Section~\ref{sec:proof-main-thm}.

\section{Applications}\label{sec:application}

We here discuss the applications of Theorem~\ref{thm:null-controllability} mentioned in the previous section for the particular
case where each $A(x)$ is the identity matrix, that is, $H_\Omega^\bullet = -\Delta_\Omega^\bullet + V$. These applications draw
upon null-controllability results from the works~\cite{EV18,NTTV18-preprint} on cubes and the full space. As considered there, we
also take control sets of the form $\omega = \Omega \cap S$, where $S \subset \RR^d$ is some measurable set with certain geometric
properties, namely a thick set or an equidistributed set. We treat the two types of sets in the subsections below separately.

The general strategy for our applications is to build from the given set $S$ a symmetric set $\tilde{S}$ with respect to the
reflection such that $\Omega \cap \tilde{S} = \Omega \cap S$. After having verified corresponding geometric properties of
$\tilde{S}$, this new set yields a suitable control set $\tilde{\omega} = \tilde{\Omega} \cap \tilde{S}$ for the symmetrized system
on $\tilde{\Omega}$, where a null-controllability result is available. Applying Theorem~\ref{thm:null-controllability} then gives
the desired null-controllability result for the original system on $\Omega$. In summary, our main task here is to establish the
geometric properties of the symmetrized set $\tilde{S}$ from those of the given set $S$.

Recall that by rotation of the whole system Theorem~\ref{thm:null-controllability} can be applied for a reflection with respect to
any hyperplane in $\RR^d$. Since the Laplacian is rotation invariant, only the potential $V$ then gets rotated, but remains
essentially bounded. In this context, we remind the reader that $M\colon\RR^d\to\RR^d$ denotes the reflection with respect to the
hyperplane $\{0\}\times\RR^{d-1}$. Furthermore, for the rest of this section, we use the following notation: 
\begin{itemize} 
 \item $H_{\theta}:=\{(x_1,\ldots, x_d)\in\RR^d \, \colon \, x_2<(\tan\theta) x_1\}$ for $0 \le \theta < \pi/2$;
 \item $M_{\theta} \colon\RR^d\rightarrow \RR^d$ denotes the reflection with respect to $\partial H_{\theta}$;
 \item $\Lambda_L^d := (0,L)^d$ for $L>0$.
\end{itemize}

\subsection{Null-controllability from thick sets}\label{subsec:thick} 

Let $V=0$, and consider the system
\begin{equation}\label{eq:heat} 
 \partial_t u(t) - \Delta^\bullet_\Omega u(t) = \chi_{\omega}v(t) \quad\text{ for }\quad 0<t<T,\quad u(0)=u_0,
\end{equation}
with $u_0\in L^2(\Omega)$ and $v\in L^2((0,T),L^2(\Omega))$. Moreover, assume that the control set $\omega$ is of the form
$\omega=\Omega\cap S$ with some $S\subset\RR^d$ that is $(\gamma,a)$-thick in the sense of the following definition. 

\begin{defin}\label{defin:thickness}
 A measurable set $S \subset \RR^d$ is called~\emph{thick} if there exist $\gamma\in (0,1]$ and $a=(a_1,\ldots,a_d)\in(0,+\infty)^d$
 such that 
 \[
  \abs{S\cap (x+[0,a_1]\times\ldots\times [0,a_d])}\geq \gamma\prod_{j=1}^d a_j \quad \forall \, x\in\RR^d,
 \]
 where $\abs{\,\cdot\,}$ denotes the Lebesgue measure. In this case, $S$ is also referred to as $(\gamma, a)$-thick to emphasise the
 parameters.
\end{defin}

The above definition has played a crucial role in a recent development on the null-controllability of the heat equation 
on $\RR^d$.
In \cite{EV18}, see also~\cite{WWZZ19}, it is shown that thickness of $S$ is a necessary and sufficient condition for the heat
equation~\eqref{eq:heat} on $\Omega=\RR^d$ to be null-controllable.
Moreover, in~\cite{EV18} the authors give an explicit estimate of the control cost 
in dependence of the thickness parameters of the set $S$ and the time $T$.
In the same paper the authors also consider the heat equation~\eqref{eq:heat} on the cube $\Omega=\Lambda_L^d$ with control set
$\omega=\Omega \cap S$, where $S$ is $(\gamma,a)$-thick with $a_j\leq L$ for all $j=1,\ldots, d$. In this case, they show that
null-controllability holds in any time $T>0$ with a bound on the control cost of the same form as for the full space case and
independent of the scale $L$.

In~\cite{NTTV18-preprint} these bounds have been strengthened in both time and geometric parameters dependency and the authors have
shown that they are close to optimality in certain asymptotic regimes, see \cite[Section 4]{NTTV18-preprint}. 
For both cases $\Omega\in\{\RR^d,\Lambda_L^d\}$ they can be written as
\begin{equation}\label{bounds:control-cost-thick}
 C_T\leq \frac{1}{\sqrt{T}}\left(\frac{K^d}{\gamma}\right)^{Kd/2}\exp\left(\frac{K\norm{a}{1}^2\ln^2(K^d/\gamma)}{2T}\right),
\end{equation}
where $K>0$ is a universal constant and $\norm{a}{1}=a_1+\dots+a_d$, see~\cite[Theorem 3.7]{NTTV18-preprint}.

We show below that system~\eqref{eq:heat} on half-spaces, positive orthants, and sectors of angle $\pi/2^n$, $n\geq 2$, is
null-controllable with a control cost bound fully or partially consistent with~\eqref{bounds:control-cost-thick}. Moreover, we
obtain null-control\-lability and explicit control cost bounds for the same system on isosceles right-angled triangles and
corresponding prisms. We once again emphasise that the results for the listed unbounded domains are new, even though the heat
equation on half-spaces and positive orthants with Dirichlet boundary conditions could also be treated with the techniques
from~\cite{SV20} (see also Remark~\ref{rmk:dirichlet-bc} below), while the novelty of the result on triangles and corresponding
prisms is the explicitness of the control cost bound on the model parameters. 

In order to enter the setting of the main theorem, let $\tilde{\Omega}\in\{\RR^d,\Lambda_L^d\}$ and
$\tilde{\omega}=\tilde{\Omega}\cap\tilde{S}$ with some $(\tilde{\gamma},\tilde{a})$-thick set $\tilde{S}\subset\RR^d$, and consider
the system
\begin{equation}\label{eq:heat:tilde} 
 \partial_t \tilde{u}(t) - \Delta^\bullet_{\tilde\Omega} \tilde{u}(t) = \chi_{\tilde{\omega}}\tilde{v}(t) \quad\text{ for }\quad 0<t<T,\quad \tilde{u}(0)=\tilde{u}_0,
\end{equation}
with $\tilde{u}_0\in L^2(\tilde{\Omega})$ and $\tilde{v}\in L^2((0,T),L^2(\tilde{\Omega}))$.  

We need the following easy geometric lemma, parts of which are already contained in~\cite{EV18}.

\begin{lemma}\label{lemma:thickness}
Let $S\subset\RR^d$ be $(\gamma, a)$-thick.
\begin{itemize} 
\item [(a)] Let $S'=S\cap((0,+\infty)\times\RR^{d-1})$. Then, the set $\tilde{S}=S'\cup M(S')$ is
$(\tilde{\gamma}, \tilde{a})$-thick with $\tilde{\gamma}=\gamma/2$ and $\tilde{a}=(2a_1,a_2,\ldots, a_d)$.

\item [(b)] The set $\tilde{S}=\{(x_1,\dots,x_d) \,\colon\, (\abs{x_1},\dots,\abs{x_d})\in S\}$ is $(\gamma/2^d, 2a)$-thick, where
$2a=(2a_1, \ldots, 2a_d)$. 
 
\item [(c)] Let $S'=S\cap H_{\theta}$. Then, $\tilde{S}=S'\cup M_{\theta}(S')$ is a $(\tilde{\gamma}, \tilde{a})$-thick set with
parameters $\tilde{\gamma}=\frac{\gamma a_1 a_2}{4(a_1^2+a_2^2)}$ and 
$\tilde{a}=(2\sqrt{a_1^2+a_2^2},2\sqrt{a_1^2+a_2^2}, a_3,\ldots, a_d)$.
\end{itemize}
\end{lemma}

\begin{proof}
We abbreviate $Q_a:=[0,a_1]\times\dots\times[0,a_d]$, $a=(a_1,\dots,a_d)\in(0,+\infty)^d$.

For part~(a) we need to show that
\[
\abs{\tilde{S}\cap (x+ Q_{\tilde a})}\geq \frac{\gamma}{2}2a_1\prod_{j=2}^d a_j = \tilde{\gamma}\abs{Q_{\tilde a}} \qquad \forall \, x\in\RR^d.
\]

Let $x\in\RR^d$. 
There is $y\in \RR^d$ such that $y+Q_a\subset (x+Q_{\tilde a})\cap([0,+\infty)\times\RR^{d-1})$ or
$M(y+Q_a)\subset (x+Q_{\tilde a})\cap((-\infty,0])\times\RR^{d-1})$, cf.~Figure~\ref{fig:thick}\,(a). 
In the first case, we have
\[
\abs{\tilde{S}\cap (x+Q_{\tilde a})}  \geq \abs{S'\cap (y+ Q_a)} = \abs{S\cap (y+ Q_a)}\geq \gamma \prod_{j=1}^d a_j = \tilde{\gamma}\abs{Q_{\tilde a}}.
\]

In the second case, we have
\[
\abs{\tilde{S}\cap (x+Q_{\tilde a})}  \geq \abs{M(S')\cap M(y+ Q_a)} = \abs{S\cap (y+ Q_a)}\geq \tilde{\gamma}\abs{Q_{\tilde a}},
\]
which completes the proof of part~(a).

For part~(b), we observe that $\tilde{S}$ can be obtained from $S\cap[0,+\infty)^d$ by successive reflection with respect to all
coordinate axes. In this regard, part~(b) follows by analogous arguments as in part~(a), cf.~Figure~\ref{fig:thick}\,(b); for more details,
see also the proof of~\cite[Theorem~4]{EV18}.

Finally, we prove part~(c). Let $x\in \RR^d$. Then, there exists $y\in\RR^d$ such that
$y+Q_a\subset (x+Q_{\tilde a})\cap \overline{H_{\theta}}$ or
$M_{\theta}(y+Q_a)\subset (x+Q_{\tilde a})\cap (\RR^d\setminus H_{\theta})$.
This follows from the fact that $x+Q_{\tilde a}$ contains the cylinder 
$x'+\big(B_2(0,\sqrt{a_1^2+a_2^2})\bigtimes_{j=3}^d [0,a_j]\big)$
for some $x'\in\RR^d$, where $B_2(0,r)$ stands for the Euclidean 2-dimensional ball centred at zero with radius $r$, and such a
cylinder contains a parellelepiped of type $y+ Q_a$ as well as its reflection with respect to $M_\theta$,
cf.~Figure~\ref{fig:thick}\,(c). The claim now follows by analogous calculations as in part~(a), taking into account that 
$\gamma \abs{Q_a} = \frac{\gamma a_1a_2}{4(a_1^2+a_2^2)}\abs{Q_{\tilde a}} = \tilde{\gamma}\abs{Q_{\tilde a}}$.
\end{proof}

\begin{figure}[htb]\label{fig:thick}
\centering
\begin{tikzpicture}[scale=0.6]
\begin{scope}[xshift=-12cm]
\node at (-3.5,4) {$(a)$};
\draw (-1,1) rectangle (2,1.5);  
\filldraw[color=black, fill=black!5] (0.5,1) rectangle (2,1.5);

\draw[<->] (0.5,0.8) -- (2,0.8); \node at (1.2,0.6) {\tiny$a$};
\draw[<->] (-1, 0.4) -- (2,0.4); \node at (0.5,0.2) {\tiny$2a$};

\draw (-2.5,2) rectangle (0.5,2.5);  
\filldraw[color=black, fill=black!5] (-2.5,2) rectangle (-1,2.5);

\draw[->] (-3.5,0) -- (3.5,0); 
\node at (3.5, -0.3) {\tiny$x_1$};
\draw[->] (0,-1) -- (0,4); 
\node at (0.2, 4.1) {\tiny$x_2$};
\end{scope}
\begin{scope}[yshift=-1cm]
\node at (-3.5,5) {$(b)$};
\draw (-1,0.7) rectangle (2,2.7);
\draw[dashed] (-1,1.7) -- (2, 1.7);
\draw[dashed] (0.5, 0.7) -- (0.5,2.7);
\filldraw[color=black, fill=black!5] (0.5, 1.7) rectangle (2, 2.7);

\draw[<->] (0.5, 2.85) -- (2, 2.85); \node at (1.2, 3) {\tiny$a_1$};
\draw[<->] (-1,3.3)--(2, 3.3); \node at (0.5, 3.5) {\tiny$2a_1$};
\draw[<->] (2.2, 1.7)-- (2.2, 2.7); \node at (2.5, 2.3) {\tiny$a_2$};
\draw[<->] (2.8, 0.7) -- (2.8, 2.7); \node at (3.2, 2) {\tiny$2a_2$};
 
\draw[->] (-3.5,1) -- (3.5,1); 
\node at (3.5, 0.7) {\tiny$x_1$};
\draw[->] (0,0) -- (0,5); 
\node at (0.2, 5.1) {\tiny$x_2$};
\end{scope}
\begin{scope}[xshift=-5.5cm,yshift=-6.5cm]
\node at (-3.5,4) {$(c)$};
\draw (-3,-1.5) -- (3.5, 1.75);
\filldraw[color=black, fill=black!5] (0,0) rectangle (2,-1);
\draw[<->] (0,-1) -- (2,-1);
\draw[<->] (2,-1) -- (2,0);
\node at (0.8, -1.2) {\tiny$a_1$};
\node at (1.7, -0.5) {\tiny$a_2$};

\draw[dashed] (0,0) -- (2,-1);
\draw (0,0) -- (-0.8,0.6);
\draw (-0.8,0.6) -- (0.4,2.2 );
\draw (0.4, 2.2) -- (1.2,1.6);
\draw (1.2,1.6) -- (0,0);
\draw[dashed] (0,0) -- (0.4,2.2);

\draw (0,0) circle (2.2360cm);
\draw (-2.2360, 2.2360) rectangle (2.2360,-2.2360);
\draw[<->] (-2.2360, 2.4) -- (2.2360, 2.4); 
\node at (-1.4, 2.75) {\tiny$2(a_1^2+a_2^2)^{1/2}$};

\draw[->] (-3.5,0) -- (3.5,0); 
\node at (3.5, -0.3) {\tiny$x_1$};
\draw[->] (0,-3) -- (0,4); 
\node at (0.2, 4.1) {\tiny$x_2$};
\end{scope}
\end{tikzpicture}
\caption{Proof of Lemma~\ref{lemma:thickness}: Positions of the parallelepipeds.}
\end{figure}
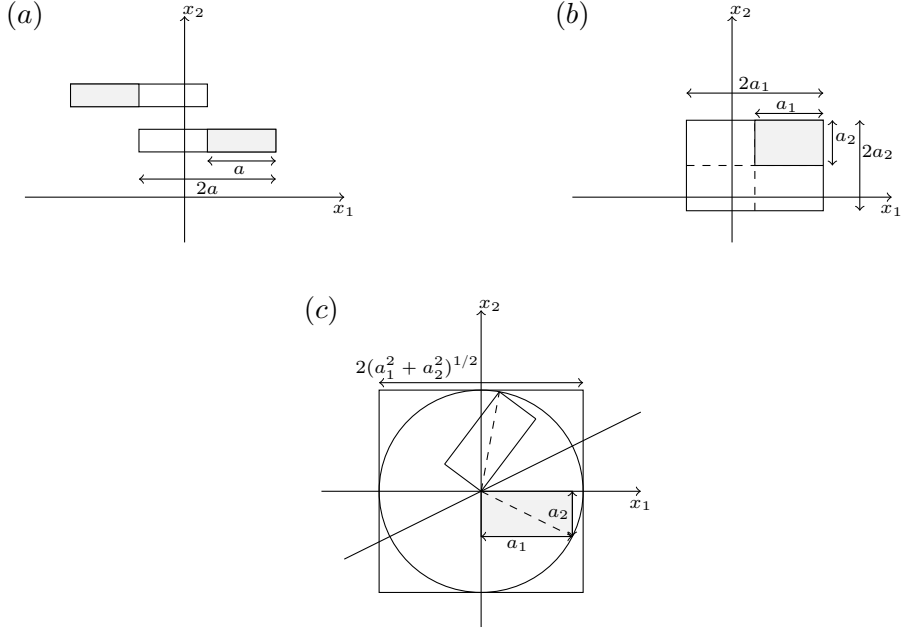

\begin{rmk}
 The choice of the parameters $\tilde\gamma$ and $\tilde a$ in part~(c) of Lemma~\ref{lemma:thickness} takes care of every possible
 orientation of reflected parallelepipeds, which are in general not parallel to coordinate axes (with the exception of $\theta=\pi/4$).
 In this regard, the angle $\theta=\pi/8$ constitutes a kind of worst case, whereas especially the angle $\theta = \pi/4$ could have
 been handled in a slightly more efficient way. We refrained from doing so for the sake of simplicity.
\end{rmk}

\begin{prop}[The heat equation on half-spaces]\label{prop:control-half-spaces}
  The system~\eqref{eq:heat} on $\Omega=(0,+\infty)\times\RR^{d-1}$ is null-controllable in any time $T>0$ with control cost 
  \begin{equation}\label{eq:control-cost-half-space}
  C_T\leq \frac{1}{\sqrt{T}}\left(\frac{2K^d}{\gamma}\right)^{Kd/2}\exp\left(\frac{K(2a_1+a_2+\ldots+a_d)^2\ln^2(2 K^d/\gamma)}{2T}\right),
  \end{equation}
  where $K>0$ is the universal constant from \eqref{bounds:control-cost-thick}.
\end{prop}

\begin{proof}
 Choose the $(\gamma/2,(2a_1,a_2\ldots,a_d))$-thick set $\tilde{S}$ as in Lemma~\ref{lemma:thickness}\,(a). Then, the heat equation~\eqref{eq:heat:tilde}
 on $\tilde{\Omega}=\RR^d$ is null-controllable in any time $T>0$ from the control set $\tilde{\omega}=\tilde{S}$. The associated bound on the
 control cost $\tilde{C}_T$ from~\eqref{bounds:control-cost-thick} now reads as in~\eqref{eq:control-cost-half-space}.
 
 Since $\tilde{\Omega}=\Omega\cup M(\Omega)\cup \Gamma$ with $\Gamma=\{0\}\times\RR^{d-1}$ and $\tilde{\omega}=\omega\cup M(\omega)$,
 the claim follows by Theorem \ref{thm:null-controllability}.
\end{proof}

\begin{prop}[Heat equation on positive orthants]\label{prop:control-positive-orthants}
 The system~\eqref{eq:heat} on $\Omega=(0,+\infty)^d$ is null-controllable in any time $T>0$ with control cost 
  \begin{equation}\label{eq:control-cost-orthants}
  C_{T} \leq \frac{1}{\sqrt{T}}\left(\frac{(2K)^d}{\gamma}\right)^{Kd/2}\exp\left(\frac{ 4 K\norm{a}{1}^2\ln^2((2K)^d/\gamma)}{2T}\right),
  \end{equation}
  where $K>0$ is the universal constant from \eqref{bounds:control-cost-thick}.
\end{prop}

\begin{proof}
 Choose the $(\gamma/2^d,2a)$-thick set $\tilde{S}$ as in Lemma~\ref{lemma:thickness}\,(b). Then, the heat equation~\eqref{eq:heat:tilde}
 on $\tilde{\Omega}=\RR^d$ is null-controllable in any time $T>0$ from the control set $\tilde{\omega}=\tilde{S}$.
 The associated bound on the control cost $\tilde{C}_T$ now reads as in~\eqref{eq:control-cost-orthants}.

 Since the set $\tilde{S}$ is symmetric with respect to every coordinate axis, the claim now follows by successively applying
 Theorem~\ref{thm:null-controllability} for the reflections with respect to the coordinate axes. This step by step leads to
 null-control\-lability of the system~\eqref{eq:heat} on $\Omega = (0,+\infty)^j \times \RR^{d-j}$, $j = 1,\dots,d$, with control
 set $\omega = ((0,+\infty)^j \times \RR^{d-j}) \cap \tilde{S}$. Each control cost does not exceed the one from the previous step
 and, thus, the corresponding bound reads as in~\eqref{eq:control-cost-orthants}. Taking into account that
 $(0,+\infty)^d \cap \tilde{S} = (0,+\infty)^d \cap S$, the final step $j = d$ then proves the claim.
\end{proof}

\begin{rmk}\label{rmk:dirichlet-bc}
 Control cost estimates on the half-space and the positive orthant can alternatively be obtained by means of the exhaustion approach
 studied in~\cite{SV20}. In this case, the corresponding bounds will not incorporate any change in the thickness parameters as the
 above bounds do. However, the considerations from~\cite{SV20} currently allow to apply this approach only in the case of Dirichlet
 boundary conditions.
\end{rmk}

\begin{prop}[Heat equation on a sector of angle $\pi/2^n$, $n\geq 2$]\label{prop:control-sectors}
 The system~\eqref{eq:heat} on the sector $\Omega = \{ (x,y)\in(0,+\infty)^2\,\colon\, y<(\tan(\pi/2^n))x \}$ of angle
 $\pi / 2^n$, $n\geq 2$, is null-controllable in any time $T>0$ with control cost 
 \begin{equation}\label{eq:control-cost-sector}
	C_T
	\leq 
	\frac{1}{\sqrt{T}} \Bigl( \frac{2^{3n-4}K^2}{\tilde\gamma} \Bigr)^K
	 \exp \Bigl( \frac{2^{3n-4}K\norm{\tilde a}{1}^2\ln^2(2^{3n-4}K^2/\tilde\gamma)}{2T} \Bigr),
 \end{equation}
 where
 \[
  \tilde\gamma = \frac{\gamma a_1a_2}{4(a_1^2+a_2^2)},
  \quad
  \tilde a = \textstyle(2\sqrt{a_1^2+a_2^2},2\sqrt{a_1^2+a_2^2}),
 \]
 and $K>0$ is the universal constant from \eqref{bounds:control-cost-thick}.
\end{prop}

\begin{proof}
 We prove the claim by induction on $n$. For $n=2$, we choose the $(\tilde\gamma,\tilde a)$-thick set $\tilde{S}$ as in
 Lemma~\ref{lemma:thickness}\,(c) with $d=2$ and $\theta = \pi/4$. Then, the heat equation~\eqref{eq:heat:tilde} on the orthant
 $\tilde{\Omega}=(0,+\infty)^2$ is null-controllable in any time $T>0$, and the bound on the associated control cost from
 Proposition~\ref{prop:control-positive-orthants} reads as in~\eqref{eq:control-cost-sector} for $n=2$. The claim for $n=2$ then
 follows by applying Theorem~\ref{thm:null-controllability} with respect to the reflection $M_{\pi/4}$.
 
 Now, suppose that the claim holds for some $n \ge 2$. Choose the $(\tilde\gamma,\tilde a)$-thick set $\tilde S$ as in
 Lemma~\ref{lemma:thickness}\,(c) with $d=2$ and $\theta=\pi/2^{n+1}$. Since $\tilde a = (\tilde a_1,\tilde a_2)$ is a multiple of
 $(1,1)$, we observe that
 \[
  \frac{\tilde\gamma\tilde a_1 \tilde a_2}{4({\tilde a_1}^2+{\tilde a_2}^2)} = \frac{\tilde\gamma}{8}
  \quad\text{ and }\quad
  \textstyle(2\sqrt{{\tilde a_1}^2 + {\tilde a_2}^2},2\sqrt{{\tilde a_1}^2 + {\tilde a_2}^2}) = 2\sqrt{2}\tilde a.
 \]
 Thus, by the induction hypothesis, the system~\eqref{eq:heat:tilde} with $\tilde\Omega$ being the sector of angle $\pi/2^n$ is
 null-controllable in any time $T$, and the corresponding bound on the control cost reads as in~\eqref{eq:control-cost-sector}, but
 with $\tilde\gamma$ and $\tilde a$ replaced by $\tilde\gamma/8$ and $2\sqrt{2}\tilde a$, respectively. The claim for $\Omega$ the
 sector of angle $\pi/2^{n+1}$ therefore follows by applying Theorem~\ref{thm:null-controllability} with respect to the reflection
 $M_{\pi/2^{n+1}}$. This concludes the proof.
\end{proof}

For the last proposition of this subsection, we suppose that the thickness parameter $a=(a_1,\dots,a_d)$ additionally satisfies
$2\sqrt{a_1^2+a_2^2}\leq L$ and $a_j\leq L$ for $j\in\{3,\ldots, d\}$.

\begin{prop}[Heat equation on triangles and triangular prisms]\label{prop:heat-triangle-prism}
  Let
  \[
  \cT_L:=\{(x, y)\in\Lambda_L^2\,\colon\, y< x\}
  \quad\text{ and }\quad
  \cP_L:=\cT_L\times (0,L).
  \] 
  
  Then, the system~\eqref{eq:heat} on $\Omega\in\{\cT_L,\cP_L\}$ is null-controllable in any time $T>0$ with control cost
  \begin{equation}\label{eq:control-cost-triangle}
  C_{T}\leq 
  \frac{1}{\sqrt{T}}\left(\frac{K^d }{\tilde\gamma}\right)^{Kd/2}\exp\left(\frac{K\norm{\tilde a}{1}^2\ln^2(K^d/\tilde\gamma)}{2T}\right),
  \end{equation}
  where $\tilde\gamma$ and $\tilde a$ are as in part~(c) of Lemma~\ref{lemma:thickness} and $K>0$ is the universal constant
  from~\eqref{bounds:control-cost-thick}. Here, $d=2$ if $\Omega=\cT_L$ and $d=3$ if $\Omega=\cP_L$.
\end{prop}

\begin{proof}
 This is proved analogously to the case $n=2$ in Proposition~\ref{prop:control-sectors} with $\tilde\Omega = \Lambda_L^d$ and the
 corresponding control cost bound from~\eqref{bounds:control-cost-thick}.
\end{proof}

\begin{rmk}\label{rmk:product-spaces}
In the recent work~\cite{Egi18}, it has been shown that the system~\eqref{eq:heat} on the strip
$\Omega=\Lambda_L^{d-1}\times\RR$ is null-controllable in any time $T>0$ if and only if $S$ is a thick set 
(which can be arbitrarily changed outside the strip), and an explicit control cost bound has been provided. 

With similar arguments as in the propositions
above, one can infer null-controllability results for system~\eqref{eq:heat} on the product spaces 
$\Lambda_L^{d-1}\times (0,+\infty)$, $\cT_L\times\RR$, and $\cT_L\times (0,+\infty)$ by means of Theorem~\ref{thm:null-controllability}. 
\end{rmk}

\begin{rmk}\label{rmk:fractionalLaplacian-BanachSpaces}
 (a) Null-controllability results and corresponding bounds on the control cost for fractional heat equations, that is,
 system~\eqref{eq:heat:tilde} with $(-\Delta_{\tilde\Omega}^\bullet)^\theta$, $\theta>\frac{1}{2}$, instead of
 $-\Delta_{\tilde\Omega}^\bullet$, have also been investigated in~\cite[Theorem~4.10]{NTTV18-preprint} on
 $\tilde{\Omega}\in\{\RR^d,\Lambda_L^d\}$ with a control cost bound of the form~\eqref{bounds:control-cost-thick} but with
 $\norm{a}{1}^2\ln^2(K^d/\gamma)/T$ replaced by $(\norm{a}{1}\ln(K^d/\gamma))^{2\theta/(2\theta-1)}/T^{1/(2\theta-1)}$.
 
 Although fractional Laplacians are formally not of divergence form as the operators considered in
 Theorem~\ref{thm:null-controllability}, the more general considerations in Section~\ref{sec:abstract-result} below nevertheless
 allow us to obtain corresponding results on the domains considered in
 Propositions~\ref{prop:control-half-spaces}--\ref{prop:heat-triangle-prism}, cf.~Remark~\ref{rmk:modification}\,(a) below.

 (b) Recent developments on null-controllability for parabolic equations in Banach spaces~\cite{BGST20,GST19} together with the
 abstract considerations in Section~\ref{sec:abstract-result} also open the way towards similar results for certain strongly
 elliptic differential operators with constant coefficients on $L^p$, $1<p<\infty$; a bound on the associated control cost on
 $\RR^d$ of a form close to~\eqref{bounds:control-cost-thick} is given in~\cite[Theorem~3.6]{BGST20}, see
 also~\cite[Corollary~4.6]{GST19}. We briefly demonstrate how to infer a corresponding result for the $L^p$-Laplacian on the
 half-space $(0,+\infty) \times \RR^{d-1}$ in Remark~\ref{rmk:modification}\,(b) below.
\end{rmk}

\subsection{Null-controllability from equidistributed sets}\label{subsec:equidistributed}

Let $V\in L^\infty(\Omega)$ be real-valued, and consider the heat-like system
\begin{equation}\label{eq:heat-like} 
 \partial_t u(t) + (- \Delta^\bullet_\Omega+ V) u(t) = \chi_{\omega}v(t) \quad\text{ for }\quad 0<t<T,\quad u(0)=u_0,
\end{equation}
with $u_0\in L^2(\Omega)$ and $v\in L^2((0,T),L^2(\Omega))$. Moreover, let the control set $\omega$ be of the form
$\omega=\Omega\cap S$ with some $S\subset\RR^d$ that is $(G,\delta)$-equidistributed in the sense of the following definition.

\begin{defin}\label{defin:equidistributed_set}
 Let $G>0$ and $\delta\in (0,G/2)$. A measurable set $S\subset\RR^d$ is called $(G,\delta)$-\emph{equidistributed} if 
\[
 S \supset \bigcup_{j\in\ZZ^d}  B(z_j, \delta)
\]
for some sequence $(z_j)_{j\in\ZZ^d}\subset\RR^d$ with $B(z_j,\delta)\subset \left(Gj+\Lambda_G^d \right)$ for all $j\in\ZZ^d$,
where $B(z_j,\delta)$ denotes the open ball in $\RR^d$ of radius $\delta$ around $z_j$.
\end{defin}

It is shown in~\cite{NTTV18-preprint} that system \eqref{eq:heat-like} is null-controllable in any time $T>0$ on domains of the
form $\Omega = \bigtimes_{j=1}^d (a_j,b_j)$ with $a_j,b_j\in\RR \cup \{\pm\infty\}$, $a_j < b_j$, such that
$\Lambda_G \subset \Omega$. A corresponding bound on the control cost is given by
\begin{equation}\label{bounds:control-cost-equidistributed}
 C_T
 \leq
 \left( \frac{\delta}{G} \right)^{-D\cdot(1+G^{4/3}\norm{V}{\infty}^{2/3})}\frac{D}{\sqrt{T}}
 \exp\left( \frac{DG^2\ln^2(\delta/G)}{2T} + \norm{V_-}{\infty}T \right),
\end{equation}
where $V_- = \max\{-V,0\}$ denotes the negative part of $V$ and $D=D(d)$ depends only on the dimension,
cf.~\cite[Theorem~4.11]{NTTV18-preprint} and~\cite[Proposition~2.4 and Corollary~2.5]{SV20}.

Below we extend the above result to sectors of angle $\pi/2^n$, $n\geq 2$, isosceles right-angled triangles, and corresponding prisms. 
A treatment of some related product spaces would also be possible, but we will not do this here for simplicity. 
In the framework of the main theorem, let
$\tilde\Omega\in\{\Lambda_L^d, (0,+\infty)^d\}$, $L\ge 2G$, and $\tilde\omega=\tilde\Omega\cap \tilde{S}$ with some
$(2G, \delta)$-equidistributed set for $\tilde{S}$. As in the previous subsection, we use the above mentioned results for the system 
\begin{equation}\label{eq:heat-like-tilde} 
 \partial_t \tilde{u}(t) + (- \Delta^\bullet_{\tilde\Omega}+ \tilde{V}) \tilde{u}(t) = \chi_{\tilde\omega}\tilde{v}(t) \quad\text{ for }\quad
 0<t<T,\quad \tilde{u}(0)=\tilde{u}_0,
\end{equation}
with $\tilde{u}_0\in L^2(\tilde\Omega)$ and $\tilde{v}\in L^2((0,T),L^2(\tilde\Omega))$ to infer null-controllability and related
control cost bounds on the desired domains. 

Analogously to Lemma~\ref{lemma:thickness} in the case of thick sets, we first need the following geometric consideration.

\begin{lemma}\label{lemma:equidistributed}
	Let $S\subset\RR^d$ be $(G,\delta)$-equidistributed, and let $S'=S \cap H_{\theta}$. Then, the set
	$\tilde{S}=S'\cup M_{\theta}(S')$ is $(4G,\delta)$-equidistributed.
\end{lemma}

\begin{proof}
We need to show that for every cube $\Lambda_j:=4Gj+\Lambda_{4G}^d$, $j\in\ZZ^d$, there exists $j'\in\ZZ^d$ such that
$Gj'+\Lambda_{G}^d\subset \Lambda_j\cap\overline{ H_\theta}$ or 
$M_{\theta}(Gj'+\Lambda_{G}^d)\subset \Lambda_j\cap(\RR^d\setminus H_\theta)$.

To this end, we observe that for every $j\in\ZZ^d$ there exists $k\in\ZZ^d$ such that  
$Gk+\Lambda_G^d \subset \Lambda_j\cap \overline{H_\theta}$ or $3Gk+\Lambda_{3G}^d\subset\Lambda_j\cap (\RR^d\setminus H_\theta)$,
cf.~Figure~\ref{fig:equidistributed}~(a). 
In the first case, we are done with $j'=k$. 
In the second case, the cube $3Gk+\Lambda_{3G}^d$ contains a ball of radius $\sqrt{2}G$, which, in turn, contains the reflection of
a cube $Q\subset H_\theta$ with sides of length $2G$ parallel to coordinate axes, cf.~Figure \ref{fig:equidistributed}~(b). 
This cube $Q$ must contain at least one cube of
the form $Gj'+\Lambda_G^d$ with $j'\in\ZZ^d$, for which
$M_\theta(Gj'+\Lambda_G^d) \subset 3Gk+\Lambda_{3G}^d\subset\Lambda_j\cap (\RR^d\setminus H_\theta)$. This concludes the proof.
\end{proof}
 
\begin{figure}[htb]\label{fig:equidistributed}
\centering
\begin{tikzpicture}[scale=0.5]
\begin{scope}[xshift=-12.5cm]
\node at (-2,6) {$(a)$};
\draw[->] (0,-1.5) -- (0,5.5); \node at (0.2, 5.7) {\tiny$x_2$};
\draw[->] (-1.5,0) -- (6,0); \node at (6, -0.4) {\tiny$x_1$}; 
\draw (-2,-2/3) -- (6,2); 
\draw (-2, -1/3) -- (6,1); 
\draw(-2,-4/3) -- (6,4); 
\foreach \x in {-1,0,1,2,3,4}{
	\foreach \y in {-1,0,1,2,3,4}{
		\draw[dashed] (\x,-1.5) -- (\x, 4.5); 
		\draw[dashed] (-1.5, \y) -- (4.5, \y);
	}
}

\draw[<->] (-0.3, 3) -- (-0.3,4); \node at (-0.6,3.5) {\small$G$};

\foreach \x in {0,4}{
	\foreach \y in {0,4}{
		\draw[thick] (\x,-1.5) -- (\x, 4.5); 
		\draw[thick] (-1.5, \y) -- (4.5, \y);
	}
}
\draw[thick, red] (0,4) rectangle (3,1);
\draw[thick, red] (3,1) rectangle (4,0);
\end{scope}

\begin{scope}
\node at (-5,6) {$(b)$};
\draw[->] (0,-5) -- (0,5.5); \node at (0.2, 5.7) {\tiny$x_2$};
\draw[->] (-5,0) -- (5.5,0); \node at (5.5, -0.4) {\tiny$x_1$}; 
\draw (-5,-5/2) -- (6,3); 

\foreach \x in {-4,-3,-2,-1,0,1,2,3,4}{
	\foreach \y in {-4,-3,-2,-1,0,1,2,3,4}{
		\draw[dashed] (\x,-5) -- (\x, 5); 
		\draw[dashed] (-5, \y) -- (5, \y);
	}
}

\draw[<->] (0,4.6) --(1,4.6); \node at (0.5,4.9) {\small$G$}; 
\draw[<->] (0,4.2) --(4,4.2); \node at (2.5,4.5) {\small$4G$}; 
\foreach \x in {-4,0,4}{
	\foreach \y in {-4,0,4}{
		\draw[thick] (\x,-5) -- (\x, 5); 
		\draw[thick] (-5, \y) -- (5, \y);
	}
}

\pgfmathsetmacro{\X}{-0.6};
\pgfmathsetmacro{\Y}{-1.1};
\pgfmathsetmacro{\XX}{\X+2};
\pgfmathsetmacro{\YY}{\Y-2};
\draw[red] (\X,\Y) rectangle (\XX, \YY); 
\node[red] at (\XX+0.4, (\Y-0.5) {\small$Q$};

\draw[red] (\X*3/5+\Y*4/5, \X*4/5 - \Y*3/5) -- (\XX*3/5+\Y*4/5, \XX*4/5 - \Y*3/5) --(\XX*3/5+\YY*4/5, \XX*4/5 - \YY*3/5) -- (\X*3/5+\YY*4/5, \X*4/5 - \YY*3/5) -- (\X*3/5+\Y*4/5, \X*4/5 - \Y*3/5);

\draw (\X*3/5+3/5+\Y*4/5-4/5, \X*4/5 +4/5 - \Y*3/5+3/5) circle (1.41cm);
\draw[thick,blue] (-3,3) rectangle (0,0); 
\draw[<->, blue] (-3,3.3) -- (0,3.3); \node[blue] at (-1.5, 3.6) {\small$3G$};
\end{scope}

\end{tikzpicture}
\caption{(a) For every reflection hyperplane cutting a $4G$-cell, there is a $G$-subcell in the lower or a $3G$-subcell in the upper
half-space. (b) The blue $3G$-subcell contains the reflection of the red cube $Q$.}
\end{figure}
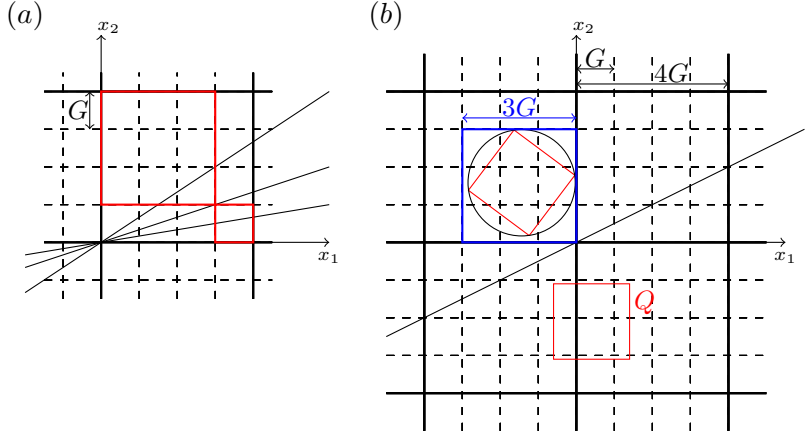

\begin{rmk}\label{rmk:equidistributed}
 If $\theta=\pi/4$, it is possible to slightly strengthen the above lemma. Indeed, in this case the resulting set $\tilde S$ is
 $(2G,\delta)$-equidistributed since the reflection $M_{\pi/4}$ maps cells of the lattice $(G\ZZ)^d$ in $\overline{H_{\pi/4}}$ to
 cells of the lattice in $\RR^d\setminus H_{\pi/4}$ and vice versa, and each cube $2Gj+\Lambda_{2G}^d$, $j\in \ZZ^d$, contains at
 least one cell of the lattice $(G\ZZ)^d$ that belongs either to $\overline{H_{\pi/4}}$ or to $\RR^d\setminus H_{\pi/4}$. However,
 for the sake of simplicity we opted for a unified statement valid for every reflection $M_\theta$.
\end{rmk}

\begin{prop}[Heat-like equation on a sector of angle $\pi/2^n$, $n\geq 2$]\label{prop:heat-eq-sector-equidistributed}
 The system~\eqref{eq:heat-like} on the sector $\Omega=\{(x,y)\in(0,+\infty)^2\,\colon\, y< (\tan(\pi/2^n))x\}$ of angle $\pi/2^n$,
 $n\geq 2$, is null-controllable in any time $T>0$ with control cost
 \begin{equation}\label{eq:control-cost-sector-equidistributed}
  C_T
  \leq
  \left(\frac{\delta}{\tilde G}\right)^{-D\cdot(1+{\tilde G}^{4/3}\norm{V}{\infty}^{2/3})}
  \frac{D}{\sqrt{T}}
  \exp\left(\frac{D{\tilde G}^2\ln^2(\delta/{\tilde G})}{2T}+\norm{V_-}{\infty}T\right),
 \end{equation}
 where $\tilde G = 4^{n-1}G$ and $D=D(2)$ is the constant from~\eqref{bounds:control-cost-equidistributed} for $d=2$.
\end{prop}

\begin{proof}
 This is proved by induction on $n$ analogously to Proposition~\ref{prop:control-sectors}, the only difference being the use of
 Lemma~\ref{lemma:equidistributed} instead of Lemma~\ref{lemma:thickness}\,(c).
\end{proof}

\begin{prop}[Heat-like equation on triangles and triangular prisms]\label{prop:heat-triangle-prism-2}
Let $L\ge 2G$, and let $\cT_L$ and $\cP_L$ as in Proposition~\ref{prop:heat-triangle-prism}. 
Then, the system~\eqref{eq:heat-like} on $\Omega\in\{\cT_L,\cP_L\}$ is null-controllable in any time $T>0$ with control cost
\[
 C_{T}
 \leq
 \left(\frac{\delta}{2G}\right)^{-R\cdot(1+(2G)^{4/3}\norm{V}{\infty}^{2/3})}\frac{R}{\sqrt{T}}
 \exp\left(\frac{R(2G)^2\ln^2(\delta/(2G))}{2T}+\norm{V_-}{\infty}T\right),
\]
where $R=\max\{D(2),D(3)\}$ with $D(d)$ from~\eqref{bounds:control-cost-equidistributed}.
\end{prop}

\begin{proof}
 Taking into account Remark~\ref{rmk:equidistributed}, this is proved analogously to the case $n=2$ in
 Proposition~\ref{prop:heat-eq-sector-equidistributed} with $\tilde\Omega=\Lambda_L^d$.
\end{proof}

\section{Abstract Result}\label{sec:abstract-result}

In this section we prove a general version of Theorem~\ref{thm:null-controllability} for more abstract control systems, which in
principle allows one to apply this result, for instance, also to other types of differential operators than discussed here so far,
such as fractional Laplacians, second order elliptic operators, and magnetic Schr\"odinger operators.

Let $\cH$ and $\cU$ be Banach spaces, $-H$ an infinitesimal generator of a $C_0$-semigroup $(S(t))_{t\geq 0}$ on $\cH$,
$B\colon\cU\to\cH$ a bounded linear operator, and $T>0$. The abstract Cauchy problem
\begin{equation}\label{eq:abstrCauchy}
\partial_t u(t) + Hu(t) = Bv(t) \quad\text{ for }\quad 0<t<T,\quad u(0)=u_0,
\end{equation}
with $u_0\in\cH$ and $v\in L^1((0,T),\cU)$ is said to be~\emph{null-controllable in time} $T>0$ with respect to $L^p((0,T),\cU)$,
$p \in [1,\infty]$, if for every initial datum $u_0\in\cH$ there is a function $v\in L^p((0,T),\cU)$ with
\begin{equation}\label{eq:defNullContr}
	S(T)u_0 + \int_0^T S(T-s)Bv(s) \dd s = 0,
\end{equation}
that is, if the mild solution to~\eqref{eq:abstrCauchy} vanishes at time $T$; see, e.g., \cite{Pazy83} for the notion of a mild
solution to abstract Cauchy problems. The associated~\emph{control cost in time $T>0$} is then defined as
\[
C_{T,p} := \sup_{\norm{u_0}{\cH}=1} \inf\{ \norm{v}{L^p((0,T),\cU)} \colon v \text{ satisfies }\eqref{eq:defNullContr} \} <\infty.
\]

The main result of this section is as follows.

\begin{thm}\label{thm:abstractResult}
	Let $\cH,\widetilde{\cH},\cU,\widetilde{\cU}$ be Banach spaces, $-H$ and $-\widetilde{H}$ infinitesimal generators of
	$C_0$-semigroups on $\cH$ and $\widetilde{\cH}$, respectively, and let $B\colon\cU\to\cH$ and
	$\widetilde{B}\colon\widetilde{\cU}\to\widetilde{\cH}$	be bounded linear operators.
	
	Suppose that there is a bounded linear operator $Y \colon \widetilde{\cH} \to \cH$ with a bounded right inverse
	$\hat Y \colon \cH \to \widetilde{\cH}$ such that
	\begin{equation}\label{eq:quasiCommutatorH}
		Y\widetilde{H} \subset HY,
	\end{equation}
	that is, $Yg\in\Dom(H)$ and $Y\widetilde{H}g = HYg$ for all $g\in\Dom(\widetilde{H})$. Furthermore, suppose that for some
	bounded	linear operator $Z \colon \widetilde{\cU} \to \cU$ one has
	\begin{equation}\label{eq:quasiCommutatorB}
		Y\widetilde{B} = B Z.
	\end{equation}
	
	If the system
	\begin{equation}\label{eq:systemTilde}
	\partial_t \widetilde{u}(t) + \widetilde{H}\widetilde{u}(t) = \widetilde{B}\widetilde{v}(t) \quad\text{ for }\quad 0<t<T,\quad
	\widetilde{u}(0)=\widetilde{u}_0,
	\end{equation}
	with $\widetilde{u}_0\in\widetilde{\cH}$ and $\widetilde{v}\in L^1((0,T),\widetilde{\cU})$ is null-controllable in time $T>0$
	with respect to $L^p((0,T),\widetilde{\cU})$ with 	control cost $\widetilde{C}_{T,p}>0$, then also the
	system~\eqref{eq:abstrCauchy} is null-controllable in time $T>0$ with respect to $L^p((0,T),\cU)$ with control cost $C_{T,p}>0$
	satisfying
	\[
		C_{T,p} \leq \norm{\hat Y}{\cH\to\widetilde{\cH}}\norm{Z}{\widetilde{\cU}\to\cU}\cdot\widetilde{C}_{T,p}.
	\]
\end{thm}

\begin{proof}
	Let $(S(t))_{t\geq 0}$ and $(\widetilde{S}(t))_{t\geq 0}$ be the $C_0$-semigroups associated to $-H$ and $-\widetilde{H}$,
	respectively. We first show that
	\begin{equation}\label{eq:semigroupQuasiCommutator}
		Y\widetilde{S}(t)
		=
		S(t)Y
		\quad\text{ for all }\ t > 0
		.
	\end{equation}
	To this end, let $\widetilde{u}_0 \in \Dom(\widetilde{\cH})$, and define the function $w \colon [0,\infty) \to \cH$ by
	$w(t) := Y\widetilde{S}(t)\widetilde{u}_0$. Since $\widetilde{S}(t)\widetilde{u}_0 \in \Dom(\widetilde{H})$ holds for all
	$t \geq 0$,	we then have $w(t) \in \Dom(H)$ and, in particular, $w(0) = Y\widetilde{u}_0 \in \Dom(H)$. Moreover,
	\[
		\partial_t w(t)
		=
		-Y\widetilde{H}\widetilde{S}(t)\widetilde{u}_0
		=
		-HY\widetilde{S}(t)\widetilde{u}_0
		=
		-Hw(t)
	\]
	for	$t > 0$. By uniqueness of the solution, see, e.g.,~\cite[Theorem~4.1.3]{Pazy83}, we conclude that $w(t) = S(t)w(0)$, that
	is, $Y\widetilde{S}(t)\widetilde{u}_0 = S(t)Y\widetilde{u}_0$ for all $t \geq 0$. Since $\Dom(\widetilde{H})$ is dense in
	$\widetilde{\cH}$ and $S(t)$ and $\widetilde{S}(t)$ are bounded operators, this proves~\eqref{eq:semigroupQuasiCommutator}.

	Let now $u_0 \in \cH$, and set $\widetilde{u}_0 := \hat Yu_0 \in \widetilde{\cH}$. Then, $Y\widetilde{u}_0 = u_0$, and by hypothesis
	there is $\widetilde{v} \in L^p((0,T),\widetilde{\cU})$ such that
	\[
		\widetilde{S}(T)\widetilde{u}_0 + \int_0^T \widetilde{S}(T-s)\widetilde{B}\widetilde{v}(s) \dd s = 0.
	\]
	Define $v \in L^p((0,T),\cU)$ by $v(s) := Z\widetilde{v}(s)$. By~\eqref{eq:quasiCommutatorB}
	and~\eqref{eq:semigroupQuasiCommutator}, we then conclude that
	\begin{multline*}
		S(T)u_0 + \int_0^T S(T-s)Bv(s) \dd s
		=
		S(T)Y\widetilde{u}_0 + \int_0^T S(T-s)BZ\widetilde{v}(s) \dd s\\
		=
		Y \left( \widetilde{S}(T)\widetilde{u}_0 + \int_0^T \widetilde{S}(T-s)\widetilde{B}\widetilde{v}(s) \dd s \right)
		=
		0
		,
	\end{multline*}
	so that system~\eqref{eq:abstrCauchy} is indeed null-controllable in time $T>0$ with respect to $L^p((0,T),\cU)$.

	Finally, using $\norm{v(s)}{\cU} \leq \norm{Z}{\widetilde{\cU}\to\cU} \norm{\widetilde{v}(s)}{\widetilde{\cU}}$ for
	$s \in (0,T)$, we observe that
	$\norm{v}{L^p((0,T),\cU)} \leq \norm{Z}{\widetilde{\cU}\to\cU} \norm{\widetilde{v}}{L^p((0,T),\widetilde{\cU})}$. Since also
	$\norm{\widetilde{u}_0}{\widetilde{\cH}} \leq \norm{\hat Y }{\cU\to\widetilde{\cU}} \norm{u_0}{\cH}$, the claimed estimate for the
	control cost $C_{T,p}$ is immediate. This completes the proof.
\end{proof}%

\begin{rmk}	\label{rmk:semigroups-as-assumption}
	Since the extension relation~\eqref{eq:quasiCommutatorH} is only used to establish the semigroup
	relation~\eqref{eq:semigroupQuasiCommutator}, the conclusion of Theorem~\ref{thm:abstractResult} still holds
	if~\eqref{eq:semigroupQuasiCommutator} is directly taken as an assumption. 

	We refrained from doing so because the extension relation between the operators gives a nice understanding of the interplay
	between the systems~\eqref{eq:abstrCauchy} and~\eqref{eq:systemTilde}.
\end{rmk}

In the particular case where $H$ and $\widetilde{H}$ are lower semibounded self-adjoint operators on Hilbert spaces $\cH$ and
$\widetilde{\cH}$, respectively, relation~\eqref{eq:semigroupQuasiCommutator} for the semigroups can be extended to far more
general functions of the operators. The following result formulates this explicitly. It is of particular interest, for instance,
when considering fractional Laplacians of the form $(-\Delta)^\theta$ with $\theta > 1/2$, see
Remark~\ref{rmk:modification}\,(a) below. The result itself is probably well known and can be proved, for instance, with a
standard reasoning using Stone's formula. We give a brief variant of this reasoning below for the convenience of the reader.

\begin{lemma}\label{lemma:functionalQuasiCommutator}
 Let $H$ and $\widetilde{H}$ be lower semibounded self-adjoint operators on Hilbert spaces $\cH$ and $\widetilde{\cH}$,
 respectively, and let $Y\colon\widetilde{\cH}\to\cH$ be a bounded linear operator such that $Y\widetilde{H} \subset HY$.
 Then, the spectral families $\EE_{\tilde H}$ and $\EE_H$ for $\tilde{H}$ and $H$, respectively, satisfy
 \[
  Y\EE_{\widetilde{H}}(\lambda) = \EE_H(\lambda)Y
  \quad\text{ for all }\ \lambda\in\RR.
 \]
 Moreover, for every Borel measurable function $\phi\colon\RR\to\RR$ the relation
 \begin{equation*}
  Y\phi(\widetilde{H}) \subset \phi(H)Y
 \end{equation*}
 holds, where $\phi(H)$ and $\phi(\widetilde{H})$ are defined by functional calculus.
 
 \begin{proof}
  It suffices to consider the case where $\cH$ and $\widetilde{\cH}$ are complex Hilbert spaces. The case of real Hilbert spaces
  can be obtained from this by complexification.
  
  In view of the relation $Y\tilde{H} \subset HY$, we have $Y(\widetilde{H}-s\pm\ii\varepsilon) \subset (H-s\pm\ii\varepsilon)Y$
  for every $s\in\RR$ and $\varepsilon>0$, so that $Y(\widetilde{H}-s\pm\ii\epsilon)^{-1}g = (H-s\pm\ii\epsilon)^{-1}Yg$ for all
  $g\in\widetilde{\cH}$ and, hence,
  \begin{equation*}
   \langle (\widetilde{H}-s\pm\ii\epsilon)^{-1}g , Y^*f \rangle_{\widetilde{\cH}}
   =
   \langle (H-s\pm\ii\epsilon)^{-1}Yg , f \rangle_{\cH}
   \quad \forall\, f\in\cH,\quad \forall \, g\in \widetilde{\cH}.
  \end{equation*}
  Stone's formula for the spectral families (see, e.g.,~\cite[Satz~8.11]{Wei00})
  then implies that
  \[
   \langle Y\EE_{\widetilde{H}}(\lambda) g , f \rangle_{\cH}
   =
   \langle \EE_{\widetilde{H}}(\lambda) g , Y^*f \rangle_{\widetilde{\cH}}
   =
   \langle \EE_{H}(\lambda) Yg , f \rangle_{\cH}
  \]
  for all $f\in\cH$, $g\in\widetilde{\cH}$, and $\lambda\in\RR$, that is, $Y\EE_{\widetilde{H}}(\lambda)=\EE_H(\lambda)Y$ for all
  $\lambda\in\RR$, which proves the first claim of the lemma.

	Now, let $\phi\colon\RR\to\RR$ be a Borel measurable function and $g\in\Dom(\phi(\widetilde{H}))$. The above then implies that
	$\langle \EE_{H}(\lambda)Yg , Yg \rangle_{\cH} \leq \norm{Y}{\widetilde{\cH}\to\cH}^2
	\langle \EE_{\widetilde{H}}(\lambda)g , g \rangle_{\cH}$, so that by functional calculus we conclude that $Yg \in \Dom(H)$ and
	\[
	 \langle Y\phi(\widetilde{H})g , f \rangle_{\cH}
	 =
	 \langle \phi(\widetilde{H})g , Y^*f \rangle_{\widetilde{\cH}}
	 =
	 \langle \phi(H)Yg , f \rangle_{\cH}
	 \qquad \forall \, f\in \cH.
	\]
	This proves the second claim and, hence, completes the proof.
 \end{proof}%
\end{lemma}

\begin{rmk}
 In the case of complex Hilbert spaces in Lemma~\ref{lemma:functionalQuasiCommutator}, also complex-valued Borel measurable
 functions $\phi$ can be considered, obviously, without any change to the proof.
\end{rmk}

\section{Proof of the main theorem}\label{sec:proof-main-thm}
We deduce Theorem~\ref{thm:null-controllability} from the abstract result Theorem~\ref{thm:abstractResult} with the choices
$\cH = \cU = L^2(\Omega)$, $\widetilde{\cH} = \widetilde{\cU} = L^2(\tilde\Omega)$, $H = H_\Omega^\bullet$,
$\widetilde{H} = H_{\tilde{\Omega}}^\bullet$, $B = \chi_\omega$, $\widetilde{B} = \chi_{\tilde{\omega}}$, and $p = 2$. Set
$\lambda^N:=1$ and $\lambda^D:=-1$, and for $\bullet\in\{D,N\}$ define
\[
 X^\bullet \colon L^2(\Omega)\rightarrow L^2(\tilde\Omega), \qquad X^\bullet f := f \oplus (\lambda^\bullet Jf),
\]
with respect to the orthogonal decomposition $L^2(\tilde\Omega)=L^2(\Omega)\oplus L^2(M(\Omega))$, where
$J\colon L^2(\Omega)\rightarrow L^2(M(\Omega))$ is defined as $J f := f\circ M$.

The adjoint $(X^\bullet)^*$ of $X^\bullet$ acts as
\begin{equation}\label{eq:Xadjoint}
 (X^\bullet)^*g = (g+\lambda^\bullet g\circ M)|_\Omega
\end{equation}
for all $g\in L^2(\tilde\Omega)$. Indeed, for $f\in L^2(\Omega)$ and $g\in L^2(\tilde\Omega)$ we have
\[
 \int_{M(\Omega)} (J f)(x)\overline{g(x)} \dd x
 = \int_{M(\Omega)} f(M(x))\overline{g(x)} \dd x = \int_{\Omega} f(x)\overline{g(M(x))} \dd x
\]
by change of variables, so that
\[
 \langle X^\bullet f,g \rangle_{L^2(\tilde\Omega)}
 = \int_{\Omega} f(x)\overline{\bigl(g(x) + \lambda^\bullet g(M(x))\bigr)} \dd x.
\]
It is then easy to see that
\begin{equation}\label{eq:rightinverseX}
 (X^\bullet)^* X^\bullet f = 2f, \qquad \forall\, f\in L^2(\Omega).
\end{equation}
Moreover, since $\tilde{\omega} = \omega \cup M(\omega)$ by construction, we have
$\chi_{\tilde\omega} \circ M = \chi_{\tilde\omega}$ and, hence,
\begin{equation}\label{eq:controlQuasiCommutator}
 (X^\bullet)^*\chi_{\tilde{\omega}}g
 =
 \bigl( \chi_{\tilde{\omega}} (g + \lambda^\bullet g\circ M) \bigr)|_\Omega
 =
 \chi_\omega (X^\bullet)^* g
 ,\qquad \forall\, g\in L^2(\tilde\Omega)
 .
\end{equation}
Hence, with the choice $Y = Z = (X^\bullet)^*$, condition~\eqref{eq:quasiCommutatorB} in Theorem~\ref{thm:abstractResult} is
satisfied, and we may take $\hat Y=X^\bullet/2$ as a bounded right inverse of $Y$ with $\norm{\hat Y}{}\norm{Z}{}=1$.

It remains to verify the extension relation~\eqref{eq:quasiCommutatorH} for the operators $H_\Omega^\bullet$ and
$H_{\tilde{\Omega}}^\bullet$, that is, $(X^\bullet)^*H_{\tilde{\Omega}}^\bullet \subset H_\Omega^\bullet(X^\bullet)^*$. 
To this end, let us first briefly recall how the
operators $H_\Omega^\bullet$ and $H_{\tilde{\Omega}}^\bullet$ are constructed. Since it is in general hard to define them directly
by their differential expression, especially if the matrix function $A$ has discontinuities, the established approach to do that is
via their quadratic forms:
Consider the form $\fh_\Omega^\bullet = \fa_\Omega^\bullet + V$ with
\[
 \fh_\Omega^\bullet[f,g] := \fa_\Omega^\bullet[f,g] + \langle Vf,g \rangle_{L^2(\Omega)},\quad
 f,g\in\Dom(\fh_\Omega^\bullet):=\Dom(\fa_\Omega^\bullet),
\]
where
\[
 \fa_\Omega^N[f,g]
 :=
 \int_\Omega \langle A(x) \nabla f(x), \nabla g(x)\rangle_{\CC^d}\dd x,
 \quad
 f,g\in\Dom(\fa_\Omega^N):= H^1(\Omega),
\]
and
\[
 \fa_\Omega^D[f,g]
 :=
 \fa_\Omega^N[f,g],
 \quad
 f,g\in\Dom(\fa_\Omega^D):=H_0^1(\Omega).
\]
The ellipticity condition~\eqref{eq:ellipticity} guarantees that this form $\fh_\Omega^\bullet$ is densely defined, lower
semibounded, and closed, cf.~\cite[Example~4.19]{Sto01}. Hence, there is a self-adjoint operator
$H_\Omega^\bullet=H_\Omega^\bullet(A,V)$ on $L^2(\Omega)$ given by
\begin{equation}\label{eq:domainH}
 \begin{aligned}
  \Dom(H_\Omega^\bullet) = \bigl\{ f\in\Dom(\fh_\Omega^\bullet)\colon \exists\,h&\in L^2(\Omega)\text{ s.t.~}\\
  &\fh_\Omega^\bullet[f,g]=\langle h,g \rangle_{L^2(\Omega)}\ \forall\,g\in\Dom(\fh_\Omega^\bullet) \bigr\}
 \end{aligned}
\end{equation}
and
\begin{equation}\label{eq:defH}
 \fh_\Omega^\bullet[f,g]
 =
 \langle H_\Omega^\bullet f,g \rangle_{L^2(\Omega)},
 \quad
 f\in\Dom(H_\Omega^\bullet),\quad g\in\Dom(\fh_\Omega^\bullet),
\end{equation}
see, e.g.,~\cite[Theorem~VI.2.6]{Kato95},~\cite[Theorem~10.7]{Schm12}, or~\cite[Theorem~VIII.15]{RS80}. The operator
$H_{\tilde\Omega}^\bullet$ on $L^2(\tilde\Omega)$ is defined completely analogous via the form
$\fh_{\tilde\Omega}^\bullet=\fa_{\tilde\Omega}^\bullet+\tilde{V}$ with
\[
 \fa_{\tilde\Omega}^N[f,g]
 :=
 \int_{\tilde\Omega} \langle \tilde{A}(x)\nabla f(x),\nabla g(x) \rangle_{\CC^d}\dd x,
 \quad
 f,g\in\Dom(\fa_{\tilde\Omega}^N):=H^1(\tilde\Omega),
\]
and
\[
 \fa_{\tilde\Omega}^D[f,g]
 :=
 \fa_{\tilde\Omega}^N[f,g],
 \quad
 f,g\in\Dom(\fa_{\tilde\Omega}^D):=H_0^1(\tilde\Omega).
\]

The desired extension relation $(X^\bullet)^* H_{\tilde{\Omega}}^\bullet \subset H_\Omega^\bullet (X^\bullet)^*$ can now be
formulated in terms of the quadratic forms as
\begin{equation}\label{eq:formQuasiCommutator}
 \fh_{\tilde{\Omega}}^\bullet[g,X^\bullet f]
 =
 \fh_\Omega^\bullet[(X^\bullet)^*g,f],
 \quad
 f \in \Dom(\fh_\Omega^\bullet),
 \quad
 g \in \Dom(\fh_{\tilde{\Omega}}^\bullet),
\end{equation}
provided that $X^\bullet f \in \Dom(\fh_{\tilde{\Omega}}^\bullet) = \Dom(\fa_{\tilde{\Omega}}^\bullet)$ and
$(X^\bullet)^*g \in \Dom(\fh_\Omega^\bullet) = \Dom(\fa_\Omega^\bullet)$; see the formal proof below.

In order to verify~\eqref{eq:formQuasiCommutator}, we first take care of the mapping properties of $X^\bullet$ and $(X^\bullet)^*$.

\begin{lemma}\label{lemma:domainX}
 \begin{enumerate}[(a)]
  \item $(X^\bullet)^*$ maps $\Dom(\fa_{\tilde\Omega}^\bullet)$ into $\Dom(\fa_\Omega^\bullet)$.
  \item $X^\bullet$ maps $\Dom(\fa_\Omega^\bullet)$ into $\Dom(\fa_{\tilde\Omega}^\bullet)$, and one has
        \[
         \bigl(\nabla(X^\bullet f)\bigr)|_\Omega = \nabla f \quad\text{ and }\quad
         \bigl(\nabla(X^\bullet f)\bigr)|_{M(\Omega)} = \lambda^\bullet U(\nabla f)\circ M
        \]
        for $f\in\Dom(\fa_\Omega^\bullet)$ with $U = \diag(-1,1,\dots,1)$.
 \end{enumerate}
 
 \begin{proof}
  (a). The case $\bullet = N$ is clear by~\eqref{eq:Xadjoint}. For $\bullet = D$ let first
  $g\in C_c^\infty(\tilde\Omega)$. Then, $g-g\circ M\in C_c^\infty(\tilde\Omega)$ and $g-g\circ M=0$ on $\Gamma$, and part~(c) of
  Lemma~\ref{lem:intParts} in the appendix yields $(X^D)^*g=(g-g\circ M)|_\Omega\in H_0^1(\Omega)$. The case of general
  $g\in H_0^1(\tilde\Omega)$ follows from this by approximation.

  (b). Let $f\in\Dom(\fa_\Omega^\bullet)$ and $\varphi\in C_c^\infty(\tilde\Omega)$, and abbreviate $\alpha_1:=-1$ and $\alpha_k:=1$
  for $k\ge 2$. We then have to show that
  \begin{equation}\label{eq:Xweak}
   -\int_{\tilde\Omega} (X^\bullet f)(x)(\partial_k\varphi)(x)\dd x
   = \int_\Omega \bigl(\partial_k f\oplus \lambda^\bullet\alpha_k(\partial_k f)\circ M\bigr)(x)\varphi(x)\dd x
  \end{equation}
  for $k=1,\dots,d$. To this end, observe that $(\partial_k\varphi)\circ M=\alpha_k\partial_k(\varphi\circ M)$, so that
  \begin{equation}\label{eq:XintParts}
   \int_{\tilde\Omega}(X^\bullet f)(x)(\partial_k\varphi)(x)\dd x
   = \int_{\Omega} f(x)\partial_k(\varphi+\lambda^\bullet\alpha_k\varphi\circ M)(x)\dd x
  \end{equation}
  by change of variables. Now, for each $k\in\{1,\dots,d\}$ and $\bullet\in\{N,D\}$, parts (a) and (b) of Lemma~\ref{lem:intParts}
  allow to perform integration by parts without boundary terms on the right-hand side of~\eqref{eq:XintParts}. We therefore obtain 
  \begin{linenomath}
  \begin{align*}
   -\int_{\Omega} f(x)\partial_k(\varphi+\lambda^\bullet\alpha_k\varphi\circ M)(x)\dd x
   &= \int_{\Omega} (\varphi+\lambda^\bullet\alpha_k\varphi\circ M)(x)(\partial_k f)(x)\dd x\\
   &= \int_{\tilde\Omega} \bigl(\partial_k f\oplus\lambda^\bullet\alpha_k(\partial_k f)\circ M\bigr)(x)\varphi(x)\dd x,
  \end{align*}
  \end{linenomath}
  where the last equality follows again by a change of variables. This equality together with~\eqref{eq:XintParts} proves
  equation~\eqref{eq:Xweak} and, hence, completes the proof.
 \end{proof}%
\end{lemma}

We are finally in position to prove our main result.

\begin{proof}[Proof of Theorem~\ref{thm:null-controllability}]
 We need to verify identity~\eqref{eq:formQuasiCommutator}. To this end, let $f \in \Dom(\fh_\Omega^\bullet) = \Dom(\fa_\Omega^\bullet)$ and
 $g \in \Dom(\fh_{\tilde{\Omega}}^\bullet) = \Dom(\fa_{\tilde{\Omega}}^\bullet)$. Since
 $U(\nabla g) \circ M = \nabla(g\circ M)$ and, by definition, $\tilde{A}(x) = UA(M(x))U$ for $x \in M(\Omega)$, we obtain from
 part~(b) of Lemma~\ref{lemma:domainX} with change of variables that
 \begin{linenomath}
 \begin{multline*}
  \int_{M(\Omega)} \langle \tilde{A}(x)\nabla g(x) , \nabla (X^\bullet f)(x) \rangle_{\CC^d} \dd x\\
  =
  \lambda^\bullet \int_\Omega \langle A(x)\nabla (g\circ M)(x) , \nabla f(x) \rangle_{\CC^d} \dd x.
 \end{multline*}
 \end{linenomath}
 Thus,
 \begin{equation}
  \label{eq:quasiCommutatora}
  \begin{aligned}
  \fa_{\tilde{\Omega}}^\bullet[g,X^\bullet f]
  &=
  \int_\Omega \langle A(x)\nabla(g + \lambda^\bullet g\circ M)(x) , (\nabla f)(x) \rangle_{\CC^d} \dd x\\
  &=
  \fa_\Omega^\bullet[(X^\bullet)^*g,f]
  ,
  \end{aligned}
 \end{equation}
 where we have taken into account~\eqref{eq:Xadjoint} and part~(a) of Lemma~\ref{lemma:domainX}. Since also
 $\tilde{V} \circ M = \tilde{V}$ and, therefore, $(X^\bullet)^*\tilde{V} = V(X^\bullet)^*$, this
 proves~\eqref{eq:formQuasiCommutator}. In turn, for
 $g \in \Dom(H_{\tilde\Omega}^\bullet) \subset \Dom(\fh_{\tilde\Omega}^\bullet)$ we conclude that
 \[
  \fh_{\Omega}^\bullet[(X^\bullet)^* g,f]
  =
  \langle H_{\tilde\Omega}^\bullet g,X^\bullet f \rangle_{L^2(\Omega)}
  =
  \langle (X^\bullet)^* H_{\tilde\Omega}^\bullet g,f \rangle_{L^2(\tilde\Omega)}
 \]
 for all $f\in\Dom(\fh_{\Omega}^\bullet)$, so that $(X^\bullet)^* g\in\Dom(H_{\Omega}^\bullet)$ and
 $H_\Omega^\bullet (X^\bullet)^* g = (X^\bullet)^* H_{\tilde\Omega}^\bullet g$, which proves the desired extension
 relation~\eqref{eq:quasiCommutatorH}. Since also~\eqref{eq:quasiCommutatorB} is satisfied by~\eqref{eq:controlQuasiCommutator},
 applying Theorem~\ref{thm:abstractResult} with $\hat Y = X^\bullet / 2$ completes the proof.
\end{proof}%

\begin{rmk}\label{rmk:modification}
 (a) The above proof in particular shows the extension relation
 $(X^\bullet)^*(-\Delta_{\tilde\Omega}^\bullet) \subset (-\Delta_\Omega^\bullet)(X^\bullet)^*$. By
 Lemma~\ref{lemma:functionalQuasiCommutator} with the particular choice $\phi=(\cdot)^\theta \mathbbm{1}_{[0,\infty)}$,
 $\theta > 1/2$, we then also have
 $(X^\bullet)^*(-\Delta_{\tilde\Omega}^\bullet)^\theta \subset (-\Delta_\Omega^\bullet)^\theta (X^\bullet)^*$. This allows us to
 apply Theorem~\ref{thm:abstractResult} also in this situation, which gives an analogue of Theorem~\ref{thm:null-controllability}
 for such fractional Laplacians, as claimed in Remark~\ref{rmk:fractionalLaplacian-BanachSpaces}\,(a).
 
 (b) A modification of the above allows to consider also certain operators on $L^p$, $p\in(1,\infty)$. We demonstrate this briefly
 for the pure Laplacian on $L^p(\Omega)$, $\Omega = (0, \infty)\times\RR^{d-1}$: Set $\tilde \Omega=\RR^d$ and let
 $-\Delta_{\Omega, p}^\bullet$ and $-\Delta_{\tilde \Omega, p}$ be the realizations on the differential expression $-\Delta$ on
 $L^p(\Omega)$ and $L^p(\tilde{\Omega})$, respectively, as closed sectorial operators, see,
 e.g.,~\cite[Theorem~5.4 and Corollary~6.11]{DHP03}. These agree with $-\Delta_\Omega^\bullet = -\Delta_{\Omega,2}^\bullet$ and
 $-\Delta_{\tilde\Omega} = -\Delta_{\tilde{\Omega},2}$, 
 respectively, on common elements of their domains. Consider now the bounded
 operator
 $Y^\bullet_p\colon L^p(\tilde \Omega) \rightarrow L^p({\Omega})$, $Y_p^\bullet g = (g + \lambda^\bullet g\circ M)\vert_{\Omega}$,
 with $\norm{Y_p^\bullet}{L^p\rightarrow L^p}= 2^{1/q}$, $1/p+1/q=1$, a right inverse of which is given by
 $\hat Y^\bullet_p \colon L^p({\Omega}) \rightarrow L^p(\tilde \Omega)$,
 $\hat Y^\bullet_p f = 2^{-1}(f \oplus \lambda^\bullet f\circ M)$, with $\norm{\hat Y_p^\bullet}{L^p\rightarrow L^p}= 2^{1/p-1}$. 
 Let $g\in \mathcal{S}(\RR^d)$ be a Schwartz function. We then have
 $Y_p^\bullet g = (X^\bullet)^* g \in \Dom(-\Delta_{\Omega, p}^\bullet) \cap \Dom(\Delta_{\Omega, 2}^\bullet)$ and
 \[
  Y^\bullet_p (-\Delta_{\tilde\Omega, p})g
  =
  (X^\bullet)^*(-\Delta_{\tilde \Omega, 2})g
  =
  (-\Delta_{\Omega, 2}^\bullet) (X^\bullet)^* g = (-\Delta_{\Omega, p}^\bullet) Y^\bullet_p g
  .
 \]
 The extension relation $Y^\bullet_p (-\Delta_{\tilde\Omega, p})\subset  (-\Delta_{\Omega, p}) Y^\bullet_p$ is then obtained by approximation since $\mathcal{S}(\RR^d)$ is an operator core for $-\Delta_{\tilde\Omega, p}$. 
Theorem \ref{thm:abstractResult} can then be applied with $Y=Z= Y^\bullet_p$ and $\hat Y = \hat Y^\bullet_p$. 
\end{rmk}

\begin{rmk}\label{rmk:triangle}
 In the above considerations, the operator $X^\bullet$ models the extension of a function in $L^2(\Omega)$ to a function in
 $L^2(\tilde{\Omega})$ by reflection with respect to one hyperplane. However, in view of the general form of
 Theorem~\ref{thm:abstractResult}, also other forms of extensions are feasible. For instance, one could consider to prolong
 functions on a sector of angle $\pi/(2n)$ in the plane to functions on an orthant by successive reflections with respect to different lines.
 In a similar way, in the work~\cite{P98} the author considers the prolongation of functions on the equilateral triangle with corners
 $(0,0)$, $(0,1)$, $(0,1/\sqrt{3})$ to functions on the rectangle $(0,\sqrt{3})\times(0,1)$. One then has to prove results analogous
 to Lemma~\ref{lemma:domainX} and relation~\eqref{eq:quasiCommutatora}, allowing to infer null-controllability results on such
 sectors and the equilateral triangle from those on the orthant and the rectangle, respectively. In fact, in case of the equilateral
 triangle, these analogous results follow to some extend from the considerations in~\cite{P98} already.
\end{rmk}

\appendix

\section{An integration by parts formula}\label{sec:intParts}
The following result is probably folklore, but in lack of a suitable reference we give an elementary proof using only integration by
parts on hypercubes. We emphasize that we are not assuming any boundary regularity for the set $\tilde{\Omega}$.

\begin{lemma}\label{lem:intParts}
 Let $\tilde{\Omega}\subset\RR^d$ be an open set with $\Gamma:=\tilde{\Omega}\cap(\{0\}\times\RR^{d-1})\neq\emptyset$, and set
 $\Omega:=\tilde{\Omega}\cap((0,+\infty)\times\RR^{d-1})$. Let $f\in H^1(\Omega)$ and $\phi\in C_c^\infty(\tilde{\Omega})$.
 \begin{enumerate}
  \renewcommand{\theenumi}{\alph{enumi}}
  \item One has
        \[
         \int_\Omega \phi(x)(\partial_k f)(x) \dd x = -\int_\Omega f(x)(\partial_k\phi)(x) \dd x
        \]
        for $k\in\{2,\dots,d\}$.
  \item If, in addition, $f\in H_0^1(\Omega)$ or $\phi|_\Gamma=0$, then
        \[
         \int_\Omega \phi(x)(\partial_1 f)(x) \dd x = -\int_\Omega f(x)(\partial_1\phi)(x) \dd x.
        \]
  \item If $\phi|_\Gamma=0$, then $\phi|_\Omega\in H_0^1(\Omega)$.
 \end{enumerate}

 \begin{proof}
  Since $\supp\phi\subset\tilde{\Omega}$ is compact, there is $\epsilon>0$ with
  \[
   2\epsilon < \dist_\infty(\supp\phi,\partial\tilde{\Omega}),
  \]
  where the distance is taken with respect to the maximum norm. We now cover $\supp\phi\cap([0,+\infty)\times\RR^{d-1})$ with
  suitable cubes from an equidistant lattice of mesh size $\epsilon$. More precisely, we choose a finite subset
  $J\subset\NN\times\ZZ^{d-1}$ with $\NN=\{1,2,\dots\}$ such that the compact set
  \[
   Q := \bigcup_{j\in J} \overline{\Lambda_\epsilon(j)},\quad
   \Lambda_\epsilon(j) := \frac{\epsilon}{2}j + \Bigl(-\frac{\epsilon}{2},\frac{\epsilon}{2}\Bigr)^d,
  \]
  is a subset of $\tilde{\Omega}\cap([0,+\infty)\times\RR^{d-1})$ containing $\supp\phi\cap([0,+\infty)\times\RR^{d-1})$ and such that
  $\supp\phi\cap((0,+\infty)\times\RR^{d-1})$ lies in the interior $Q^\circ$ of $Q$. Clearly, we have
  $Q\setminus(\{0\}\times\RR^{d-1})\subset\Omega$.

  An integration by parts on each cube $\Lambda=\Lambda_\epsilon(j)\subset\Omega$, $j\in J$, yields
  \begin{equation}\label{eq:intPartCube}
   \begin{aligned}
    \int_{\Lambda} \phi(x)(\partial_kf)(x)\dd x = \int_{\Gamma_k^+} f(x)&\phi(x) \dd\sigma(x) 
      - \int_{\Gamma_k^-} f(x)\phi(x) \dd\sigma(x)\\
    &-\int_{\Lambda} f(x)(\partial_k\phi)(x)\dd x,
   \end{aligned}
  \end{equation}
  where $\Gamma_k^+$ and $\Gamma_k^-$ denote the two opposite faces of the cube $\Lambda$ (relatively open in the corresponding
  hyperplane) with outward unit normal in positive and negative direction with respect to the $k$-th coordinate axis, respectively.
  Now, each of the faces $\Gamma_k^+$ and $\Gamma_k^-$ either corresponds to the face of exactly one other cube $\Lambda_\epsilon(l)$,
  $l\in J\setminus\{j\}$, but then with the opposite direction of the corresponding outward unit normal, or the face belongs to the
  boundary of $Q$. In the latter case, the integral over the face vanishes unless the face belongs to the hyperplane
  $\{0\}\times\RR^{d-1}$, since $\supp\phi\cap((0,+\infty)\times\RR^{d-1})$ lies in the interior of $Q$. This means that after
  summing over all cubes $\Lambda_\epsilon(j)$, $j\in J$, in~\eqref{eq:intPartCube} only boundary integrals for faces in the
  hyperplane $\{0\}\times\RR^{d-1}$ can survive, and these faces have an outward unit normal in the negative direction in the first
  coordinate ($k=1$).

  Thus, after summing in~\eqref{eq:intPartCube} over all cubes $\Lambda_\epsilon(j)$, $j\in J$, in both cases (a) and (b) we have no
  remaining boundary integrals, so that
  \[
   \int_Q \phi(x)(\partial_kf)(x)\dd x = -\int_Q f(x)(\partial_k\phi)(x)\dd x,
  \]
  where for $k=1$ we may approximate $f\in H_0^1(\Omega)$ with functions from $C_c^\infty(\Omega)$ and then take the limit. Since
  $\supp\phi\cap((0,+\infty)\times\RR^{d-1})\subset Q^\circ\subset\Omega$, this proves (a) and (b).

  It remains to show that $\phi|_\Omega\in H_0^1(\Omega)$ if $\phi|_\Gamma=0$. For that, we choose an open cube
  $\tilde\Lambda\subset\RR^d$ with $\supp\phi\subset\tilde\Lambda$ and
  $\Lambda':=\tilde\Lambda\cap((0,+\infty)\times\RR^{d-1})\neq\emptyset$. Extending $\phi$ by zero on
  $\tilde{\Lambda}\setminus\tilde{\Omega}$, and taking into account that $\phi|_\Gamma=0$ by hypothesis, we have
  $\phi\in C_c^\infty(\tilde{\Lambda})$ and $\phi=0$ on $\partial\Lambda'$. Since $\Lambda'$ is convex and therefore has a Lipschitz
  boundary, see, e.g.,~\cite[Corollary~1.2.2.3]{Gri85}, this implies that $\phi|_{\Lambda'}\in H_0^1(\Lambda')$, see,
  e.g.,~\cite[Corollary~1.5.1.6]{Gri85} or~\cite[Lemma~A6.10]{Alt06}. Thus, there exists a sequence $(\varphi_k)$ in
  $C_c^\infty(\Lambda')$ such that $\varphi_k\to\phi|_{\Lambda'}$ in $H^1(\Lambda')$ as $k\to\infty$. We choose a smooth cutoff function
  $\eta\in C_c^\infty(\tilde\Omega)$ with $\eta=1$ on $\supp\phi$. Since $(\eta\phi)|_\Omega=\phi|_\Omega$, it is
  straightforward to verify that $(\eta\varphi_k)|_\Omega\to\phi|_\Omega$ in $H^1(\Omega)$ as $k\to\infty$. Since also
  $(\eta\varphi_k)|_\Omega\in C_c^\infty(\Omega)$, we conclude that $\phi|_\Omega\in H_0^1(\Omega)$. This proves (c) and, hence,
  completes the proof of the lemma.
 \end{proof}%
\end{lemma}

\section*{Acknowledgements}
The authors are grateful to Christoph Schumacher for an inspiring discussion and to Christian Seifert for a helpful correspondence.
The authors also thank the anonymous referee of an earlier version of this manuscript for a useful comment that helped to improve
this manuscript.


\end{document}